\documentclass[12pt, oneside, a4paper]{article}

\usepackage{amsmath}
\usepackage{amsfonts}
\usepackage{amssymb}
\usepackage{amsthm,mathrsfs}
\usepackage{enumerate}
\usepackage{graphicx}
\newtheorem{theorem}{Theorem}[section]
\newtheorem{lemma}[theorem]{Lemma}

\theoremstyle{definition}

\title{\textbf{Disconnected Character graphs and odd Dominating sets}}
\author{Mahdi Ebrahimi\footnote{ m.ebrahimi.math@ipm.ir}
 \\
 {\small\em  School of Mathematics, Institute for Research in Fundamental Sciences (IPM)},\\{\small\em P.O. Box: 19395--5746, Tehran, Iran}}

\date{}

\begin{document}

\maketitle

\begin{abstract}
Suppose $\Gamma$ is a finite simple graph. If $D$ is a dominating set of $\Gamma$ such that each $x\in D$ is contained in the set of vertices of an odd cycle of $\Gamma$, then we say that $D$ is an odd dominating set for $\Gamma$.
For a finite group $G$, let $\Delta(G)$ denote the character graph built on the set of degrees of the irreducible complex characters of $G$. In this paper, we show that the complement of $\Delta(G)$ contains an odd dominating set, if and only if $\Delta(G)$ is a disconnected graph with non-bipartite complement.

 \end{abstract}

\section{Introduction}
$\noindent$ Let $G$ be a finite group and $R(G)$ be the solvable radical of $G$. Also let ${\rm cd}(G)$ be the set of all character degrees of $G$, that is,
 ${\rm cd}(G)=\{\chi(1)|\;\chi \in {\rm Irr}(G)\} $, where ${\rm Irr}(G)$ is the set of all complex irreducible characters of $G$. The set of prime divisors of character degrees of $G$ is denoted by $\rho(G)$.
A useful way to study the character degree set of a finite group $G$ is to associate a graph to ${\rm cd}(G)$.
One of these graphs is the character graph $\Delta(G)$ of $G$ \cite{[I]}. Its vertex set is $\rho(G)$ and two vertices $p$ and $q$ are joined by an edge if the product $pq$ divides some character degree of $G$. We refer the readers to a survey by Lewis \cite{[M]} for results concerning this graph and related topics.

 If we know $\Delta(G)$ is disconnected, we can often say a lot about the structure of the group $G$. For instance, Lewis and White in \cite{[nsc]} proved that $\Delta(G)$ has three connected components if and only if $G=S\times A$, where $S\cong \rm{PSL}_2(2^n)$ for some integer $n\geq 2$ and $A$ is an abelian group. Also when $G$ is non-solvable and $\Delta(G)$ has two connected components, they \cite{[nsc]} showed that $G$ has normal subgroups $N\subseteq K$ so that $K/N\cong \rm{PSL}_2(p^n)$ where $p$ is a prime and $n$ is an integer so that $p^n\geq 4$. Furthermore, $G/K$ is abelian and $N$ is either abelian or metabelian. As another instance, all finite solvable groups $G$ whose character graph $\Delta(G)$ is disconnected have been completely classified by Lewis \cite{los}. In this paper, we wish to use the dominating sets of the complement of $\Delta(G)$ to describe a new characterization of disconnected character graphs with non-bipartite complement.

A dominating set for a graph $\Gamma$ with vertex set $V$ is a subset $D$ of $V$ such that every vertex not in $D$ is adjacent to at least one member of $D$. The domination number of $\Gamma$ is the number of vertices in a smallest dominating set for $\Gamma$. If $D$ is a dominating set of $\Gamma$ such that each $x\in D$ is contained in the set of vertices of an odd cycle of $\Gamma$, then we say that $D$ is an odd dominating set of $\Gamma$. Now we are ready to state our main result. \\

\noindent \textbf{Main Theorem.}  \textit{Let $G$ be a finite group. Then the following are equivalent:\\
\textbf{a)} The complement of $\Delta(G)$ contains an odd dominating set $D$. \\
\textbf{b)}  The complement of $\Delta(G)$ is non-bipartite with domination number $1$.\\
\textbf{c)} $\Delta(G)$ is a disconnected graph with non-bipartite complement.\\
Also when one of the above conditions holds, then there exists a normal subgroup $R(G)<M\leq G$ so that $G/R(G)$ is an almost simple group with socle $S:=M/R(G)\cong \rm{PSL}_2(u^\alpha)$, where $u$ is a prime, $\alpha$ is a positive integer and $u^\alpha \geq 5$.
 }

\section{Preliminaries}
$\noindent$ In this paper, all groups are assumed to be finite and all
graphs are simple and finite. The complement of a graph $\Gamma$ is denoted by $\Gamma^c$. For a finite group $G$, the set of prime divisors of $|G|$ is denoted by $\pi(G)$. Also note that for an integer $n\geq 1$,  the set of prime divisors of $n$ is denoted by $\pi(n)$.
If $H\leq G$ and $\theta \in \rm{Irr}(H)$, we denote by $\rm{Irr}(G|\theta)$ the set of irreducible characters of $G$ lying over $\theta$ and define $\rm{cd}(G|\theta):=\{\chi(1)|\,\chi \in \rm{Irr}(G|\theta)\}$. We begin with Corollary 11.29 of \cite{[isa]}.

\begin{lemma}\label{fraction}
Let $ N \lhd G$ and $\varphi \in \rm{Irr}(N)$. Then for every $\chi \in \rm{Irr}(G|\varphi)$, $\chi(1)/\varphi(1)$ divides $[G:N]$.
\end{lemma}

  \begin{lemma}\label{lw}\cite{[non]}
 Let $p$ be a prime, $f\geqslant 2$ be an integer, $q=p^f\geqslant 5$ and $S\cong \rm{PSL}_2(q)$. If $q\neq9$ and $S\leqslant G\leqslant \rm{Aut}(S)$, then $G$ has irreducible characters of degrees $(q+1)[G:G\cap \rm{PGL}_2(q)]$ and $(q-1)[G:G\cap \rm{PGL}_2(q)]$.
 \end{lemma}

\begin{lemma}\label{good}\cite{[Ton]}
Let $N$ be a normal subgroup of a group $G$ so that $G/N\cong S$, where $S$ is a non-abelian simple group. Let $\theta \in \rm{Irr}(N)$. Then either $\chi (1)/\theta(1)$ is divisible by two distinct primes in $\pi(G/N)$ for some $\chi \in \rm{Irr}(G|\theta)$ or $\theta$ is extendible to $\theta_0\in \rm{Irr}(G)$ and $G/N\cong A_5$ or $\rm{PSL}_2(8)$.
\end{lemma}

  We now state some relevant results on character graphs
needed in the next section.

\begin{lemma}\label{ote}\cite{ME}
Let $G$ be a finite group, $R(G)< M\leqslant G$, $S:=M/R(G)$ be isomorphic to $ \rm{PSL}_2(q)$, where for some prime $p$ and positive integer $f\geqslant 1$, $q=p^f$, $|\pi(S)|\geqslant 4$ and $S\leqslant G/R(G)\leqslant \rm{Aut}(S)$. Also let $\theta\in \rm{Irr}(R(G))$. If $\Delta(G)^c$ is not  bipartite,  then $\theta$ is $M$-invariant.
\end{lemma}

\begin{lemma}\label{isolated} \cite{[nsc]}
Suppose the character graph $\Delta(G)$ of a non-solvable group $G$ is disconnected. Then $\Delta(G)$ has at least an isolated vertex.
\end{lemma}

\begin{lemma}\label{bipartite}\cite{1}
Let $G$ be a solvable group. Then $\Delta(G)^c$ is bipartite.
\end{lemma}

When $\Delta(G)^c$ is not a bipartite graph, then there exists a useful restriction on the structure of $G$ as follows:

\begin{lemma}\label{cycle} \cite{AC}
Let $G$ be a finite group and $\pi$ be a subset of the vertex set of $\Delta(G)$ such that $|\pi|> 1$ is an odd number. Then $\pi$ is the set of vertices of a cycle in $\Delta(G)^c$ if and only if $O^{\pi^\prime}(G)=S\times A$, where $A$ is abelian, $S\cong \rm{SL}_2(u^\alpha)$ or $S\cong \rm{PSL}_2(u^\alpha)$ for a prime $u\in \pi$ and a positive integer $\alpha$, and the primes in $\pi - \{u\}$ are alternately odd divisors of $u^\alpha+1$ and  $u^\alpha-1$.
\end{lemma}


\section{Proof of Main Theorem}
$\noindent$In this section, we wish to prove our main result.

\begin{lemma}\label{main}
Let $G$ be a finite group. If  $\Delta(G)^c$ contains an odd dominating set $D$, then:\\
\textbf{a)} There exists a normal subgroup $R(G)<M\leq G$ so that $G/R(G)$ is an almost simple group with socle $S:=M/R(G)\cong \rm{PSL}_2(u^\alpha)$, where $u$ is a prime, $\alpha$ is a positive integer and $u^\alpha \geq 5$.\\
\textbf{b)} For every $x\in D$, there exists $\pi_x \subseteq \pi(S)$ such that $x,u\in \pi_x$ and $\pi_x$ is the set of vertices of an odd cycle in $\Delta(G)^c$.
\end{lemma}

\begin{proof}
Let $x\in D$. Then there exists $\pi_x \subseteq \rho(G)$ such that $x\in \pi_x$ and $\pi_x$ is the set of vertices of an odd cycle in $\Delta(G)^c$.
 Thus by Lemma \ref{cycle}, $N_x:=O^{\pi^\prime_x}(G)=R_x\times A_x$, where $A_x$ is abelian,
 $R_x \cong \rm{SL}_2(u_x^{\alpha_x})$ or $R_x \cong \rm{PSL}_2(u_x^{\alpha_x})$ for a prime $u_x\in \pi_x$ and a positive integer $\alpha_x$,
  and the primes in $\pi_x - \{u_x\}$ are alternately odd divisors of $u_x^{\alpha_x}+1$ and  $u_x^{\alpha_x}-1$.
  Let $\mathcal{K}:=\{R_x|\;x\in D\}$. Define $K$ as the product of all subgroups in  $\mathcal{K}$.
   We fix a subset $\{R_{x_1},\dots, R_{x_l}\}$ of $\mathcal{K}$ such that $K/Z(K)\cong R_{x_1}/Z(R_{x_1})\times \dots \times R_{x_l}/Z(R_{x_l})$. Note that $D\subseteq \pi(K/Z(K))$ and for every $1\leq i\leq l$, $S_i:= R_{x_i}/Z(R_{x_i})\cong \rm{PSL}_2(u_{x_i}^{\alpha_{x_i}})$ and $2\in \pi(S_i)$. If $l>1$, then we can see that $2\notin D$ and in $\Delta(G)$, $2$ is adjacent to all vertices in $D$. It is a contradiction as $D$ is an odd dominating set for $\Delta(G)^c$. Hence $l=1$ and $\mathcal{K}=\{R_{x_1}\}$. We set $N:=N_{x_1}$, $u:=u_{x_1}$ and $\alpha:=\alpha_{x_1}$. Let $M:=NR(G)$. Then $S:=M/R(G)\cong N/R(N)\cong \rm{PSL}_2(u^\alpha)$ is a non-abelian minimal normal subgroup of $G/R(G)$. When $S\cong A_5$, we choose $u:=5$.
   Note that for every $x\in D$, $u\in\pi_x\subseteq \pi(S)$.
 Let $C/R(G)=C_{G/R(G)}(M/R(G))$. We claim that $C=R(G)$. Suppose on the contrary that $C\neq R(G)$ and let $L/R(G)$ be a chief factor of $G$ with $L\leqslant C$. Then $L/R(G)\cong T^k$, for some non-abelian simple group $T$ and some integer $k\geqslant 1$.
 As $L\leqslant C$, $LM/R(G)\cong L/R(G)\times M/R(G) \cong S\times T^k$. Since $2\in \pi (S)\cap \pi(T) $ and $\Delta(S \times T^k)\subseteq \Delta(G)$, it is easy to see that $2\notin D$ and in $\Delta(G)$, $2$ is adjacent to all vertices in $D$ which is impossible. Therefore $G/R(G)$ is an almost simple group with socle $S=M/R(G)$. It completes the proof.
\end{proof}

\noindent Proof of Main Theorem.
 \textbf{a}$\Rightarrow$ \textbf{b)}  Using Lemma \ref{main}, there exists a normal subgroup $R(G)<M\leq G$ so that $G/R(G)$ is an almost simple group with socle $S:=M/R(G)\cong \rm{PSL}_2(u^\alpha)$, where $u$ is a prime, $\alpha$ is a positive integer and $u^\alpha \geq 5$. Since $\Delta(G)^c$ is non-bipartite, it is enough to show that $u$ is an isolated vertex for $\Delta(G)$. Suppose on the contrary that there exists $v\in \rho(G)-\{u\}$ such that $u$ and $v$ are adjacent vertices in $\Delta(G)$. Then for some $\chi\in \rm{Irr}(G)$, $uv$ divides $\chi(1)$. Let $\varphi \in \rm{Irr}(M)$ and $\theta\in\rm{Irr}(R(G))$ be constituents of $\chi_M$ and $\varphi_{R(G)}$, respectively. We claim that $u,v\notin \pi([G:M])$. On the contrary we assume that for some $y\in \{u,v\}$, $y$ divides $[G:M]$. By Lemma \ref{lw} in $\Delta(G)$, $y$ is adjacent to all vertices in $\pi(u^{2\alpha}-1)$. Note that using Lemma \ref{main} (b), $D\subseteq \pi(S)$. If $y=u$, then there is no any $\pi\subseteq \pi(S)$ with this property that $u\in \pi$ and $\pi$ is the set of vertices of an odd cycle in $\Delta(G)^c$. It is a contradiction with Lemma \ref{main} (b). If $y=v$ and $v\in D$, then we again obtain a contradiction with Lemma \ref{main} (b). Also if $y=v$ and $v\notin D$, then $v$ is adjacent to all vertices in $D$ which is a contradiction with this fact that $D$ is an odd dominating set for $\Delta(G)^c$. Hence $u,v \notin \pi([G:M])$. Thus by Lemma \ref{fraction}, $uv$ divides $\varphi(1)\in \rm{cd}(M| \theta)$. Since $\Delta(G)^c$ is non-bipartite, using Lemmas \ref{good} and \ref{ote}, $\theta$ is $M$-invariant. Therefore as $\rm{SL}_2(u^\alpha)$ is the Schur representation of $S$, we deduce that $\theta(1)$ is divisible by $u$ or $v$. Hence for some $y\in \{u,v\}$, $y$ is adjacent to all vertices in $\pi(S)-\{y\}$. It is a contradiction thus $u$ is an isolated vertex and it completes the proof of this part.\\
 \textbf{b}$\Rightarrow$\textbf{c)} Since the domination number of $\Delta(G)^c$ is equal to $1$, we have nothing to prove.\\
\textbf{c}$\Rightarrow$\textbf{a)} Since $\Delta(G)^c$ is non-bipartite, by Lemma \ref{bipartite}, $G$ is non-solvable. Thus as $\Delta(G)$ is disconnected, using Lemma \ref{isolated}, $\Delta(G)$ has an isolated vertex $u$. Hence $D:=\{u\}$ is an odd dominating set for $\Delta(G)^c$. \\
Finally, as parts \textbf{(a)}, \textbf{(b)} and \textbf{(c)} are equivalent, when one of these parts occurs, then by Lemma \ref{main} (a),  there exists a normal subgroup $R(G)<M\leq G$ so that $G/R(G)$ is an almost simple group with socle $S:=M/R(G)\cong \rm{PSL}_2(u^\alpha)$, where $u$ is a prime, $\alpha$ is a positive integer and $u^\alpha \geq 5$.

\section*{Acknowledgements}
This research was supported in part
by a grant  from School of Mathematics, Institute for Research in Fundamental Sciences (IPM).


\end{document}